\documentclass[11pt]{article}
\usepackage[T1]{fontenc}
\usepackage{lmodern}
\usepackage[margin=1in]{geometry}
\usepackage{amsmath,amssymb,amsthm,mathtools}
\usepackage{booktabs,array,longtable}
\usepackage[hidelinks]{hyperref}
\usepackage{microtype}
\hypersetup{pdftitle={A proof of the Berkovich--Dhar conjecture modulo three},
 pdfauthor={Shutao Jiang},
 pdfsubject={Coefficient signs in powers of the finite Borwein product},
 pdfkeywords={Borwein polynomial, coefficient signs, saddle point method}}
\allowdisplaybreaks
\numberwithin{equation}{section}
\newtheorem{theorem}{Theorem}[section]
\newtheorem{lemma}[theorem]{Lemma}
\newtheorem{proposition}[theorem]{Proposition}
\newtheorem{corollary}[theorem]{Corollary}
\theoremstyle{remark}

\newcommand{\ii}{\mathrm i}
\newcommand{\e}{\mathrm e}
\newcommand{\Rea}{\operatorname{Re}}
\newcommand{\Ima}{\operatorname{Im}}
\newcommand{\Li}{\operatorname{Li}_2}
\newcommand{\G}{\mathcal G}
\newcommand{\M}{\mathcal M}
\newcommand{\sgn}{\operatorname{sgn}}
\newcommand{\ind}{\mathbf 1}
\title{A proof of the Berkovich--Dhar conjecture modulo three}
\author{Shutao Jiang\\[3pt]
 {\small East China Normal University}\\
 {\small\href{mailto:51290147024@stu.ecnu.edu.cn}{\texttt{51290147024@stu.ecnu.edu.cn}}}}
\date{}
\begin{document}
\maketitle
\begin{abstract}
Let $P_n(q)=\prod_{j=1}^n(1-q^{3j-2})(1-q^{3j-1})$.
For each $p\in\{4,5,6,7,8\}$, we prove that the coefficients of
$P_n(q)^p$ in the residue class $0$ modulo $3$ are nonnegative,
and that those in residue class $2$, after zero terms are omitted,
change sign exactly once, from positive to negative. The former
coefficients are strictly positive for $5\le p\le8$.
These results establish the sign assertions of the Berkovich--Dhar
conjecture modulo three. We also determine the limiting transition
constants and a four-term expansion for the transition centres.
For $n\ge301$, the signs in residue class $2$ are determined outside an interval of
length $2\varepsilon_p/n^2$ centred at
$\alpha_p n^2+\beta_p n+\gamma_p+\delta_p/n$, where
$\varepsilon_4=1500$ and $\varepsilon_p=150$ for $5\le p\le8$.
The constants have explicit analytic definitions. A correction to the
saddle relation removes the first phase error, a higher-order expansion
resolves the cancellation of the two dominant contributions, and a
rescaled positive integral controls the small-degree range uniformly.
\end{abstract}

\section{Introduction and main results}
For a positive integer $n$, put
\begin{equation}\label{eq:product}
 P_n(q)=\prod_{j=1}^n(1-q^{3j-2})(1-q^{3j-1}),\qquad
 c_d^{(p)}(n)=[q^d]P_n(q)^p,\qquad
 z_{p,n,m}=c_{3m+2}^{(p)}(n).
\end{equation}
The classical Borwein conjectures concern the signs of coefficients
in powers of this finite product; see Andrews~\cite{Andrews}.
Wang~\cite{Wang} proved the first conjecture analytically.
Wang and Krattenthaler~\cite{WK} subsequently developed a saddle point
approach that proves the first two Borwein conjectures and gives
partial results for the cubic case. Related questions for positive
real powers of the infinite Borwein product were studied by Schlosser
and Zhou~\cite{SZ}, who obtained coefficient asymptotics and sign
results by the circle method.

For the powers $4\le p\le8$, Berkovich and Dhar
\cite[Conjecture 2.1]{BD} proposed a different pattern. The coefficients
in residue class $0$ modulo $3$ should remain nonnegative, whereas those
in residue class $2$ should pass once from nonnegative to nonpositive.
They further conjectured a limiting position for this transition after
division by $n^2$. We prove these assertions for all five powers and
give a more precise localization. The degree of $P_n^p$ is $3pn^2$;
thus $0\le m\le pn^2-1$ is the relevant range for $z_{p,n,m}$.

\begin{theorem}\label{thm:main}
Let $p\in\{4,5,6,7,8\}$ and $n\ge1$. Then
\begin{equation}\label{eq:r0}
 c_{3m}^{(p)}(n)\ge0\qquad(0\le m\le pn^2),
\end{equation}
with strict inequality if $5\le p\le8$. The sequence
$(z_{p,n,m})_{m=0}^{pn^2-1}$ has exactly one sign change, from positive
to negative, when zero terms are omitted.

Define
\begin{equation}\label{eq:T}
 T_{p,n}=\alpha_p n^2+\beta_p n+\gamma_p+\delta_p/n,
 \qquad
 \varepsilon_p=\begin{cases}1500,&p=4,\\150,&5\le p\le8,
 \end{cases}
\end{equation}
where the constants are defined in Section~\ref{sec:constants}.
For $n\ge301$ and every integer $m$ in the relevant range,
\begin{align}
 m<T_{p,n}-\varepsilon_p/n^2&\ \Longrightarrow\ z_{p,n,m}>0,
 \label{eq:signplus}\\
 m>T_{p,n}+\varepsilon_p/n^2&\ \Longrightarrow\ z_{p,n,m}<0.
 \label{eq:signminus}
\end{align}
There is at most one zero in this sequence for $n\ge301$.
\end{theorem}

The constants in the theorem are displayed in Table~\ref{tab:constants}.
Their exact definitions, rather than the displayed decimal
approximations, will be used throughout.
\begin{table}[ht]
\centering\small
\begin{tabular}{crrrr}
\toprule
$p$&$\alpha_p$&$\beta_p$&$\gamma_p$&$\delta_p$\\\midrule
4&0.74908009478851&0.08814225705231&$-0.45919776536944$&0.05471364270799\\
5&1.26928368492984&0.05575166361208&$-0.45281773147117$&0.05611656220711\\
6&1.77728903391998&0.04091445985236&$-0.44340353651555$&0.05676868123636\\
7&2.28142027960530&0.03244327609517&$-0.43629285259280$&0.05474636409158\\
8&2.78386833123840&0.02696867636641&$-0.43118823661271$&0.05163946199434\\
\bottomrule
\end{tabular}
\caption{The transition constants, rounded for display.}\label{tab:constants}
\end{table}

\begin{corollary}\label{cor:BD}
For $4\le p\le8$, let $K_{p,n}$ be the largest index with
$z_{p,n,K_{p,n}}>0$. Then
\[
 K_{p,n}=\alpha_p n^2+\beta_p n+O(1),\qquad
 \frac{K_{p,n}}{n^2}\longrightarrow\alpha_p.
\]
In particular, the sign assertions and the existence of limiting
transition constants in \cite[Conjecture 2.1]{BD} hold for all five powers.
\end{corollary}
\begin{proof}
Theorem~\ref{thm:main} gives $K_{p,n}=T_{p,n}+O(1)$.
Reciprocity, $c_d^{(p)}(n)=c_{3pn^2-d}^{(p)}(n)$, determines the
corresponding pattern in residue class $1$.
\end{proof}

The analytic values in Table~\ref{tab:constants} replace the preliminary
numerical estimates in \cite[Conjecture 2.1]{BD}. The interval in
\eqref{eq:signplus}--\eqref{eq:signminus} describes the transition before
rounding to an integer index. If it contains no integer, then
$K_{p,n}=\lfloor T_{p,n}\rfloor$; otherwise the sign of its one possible
integer term determines the last positive index.

The range of powers has a simple obstruction beyond its upper endpoint.
For every $n\ge1$,
\begin{equation}\label{eq:p9}
 [q^3]P_n(q)^p=p^2-\binom p3=\frac{p(-p^2+9p-2)}6.
\end{equation}
This is negative for every integer $p\ge9$, so the nonnegativity
assertion in residue class $0$ cannot extend to these powers.

The main analytic difficulty is a cancellation between the contributions
of two conjugate arcs. Their leading phase vanishes at an interior saddle
parameter, and absolute-value estimates alone do not determine the
coefficient signs near it. We first remove the imaginary part of the
first correction by changing the degree--saddle relation. Near the
transition we use a fixed saddle and expand through order $n^{-3}$.
Away from the compact saddle range we rescale the parameter, ultimately
comparing the major arcs with a positive real integral. The estimates
overlap in degree and leave only finite coefficient calculations.
The shared expansions and remainder bounds are developed once below;
the dependence on $p$ is kept explicit. Every terminating decimal in
an inequality denotes the corresponding rational number.

\section{The dominant arcs and transition constants}\label{sec:constants}
Put $\omega=\e^{2\pi\ii/3}$ and define, initially for real $t$,
\begin{align}
 \ell(t)&=\log(1+\e^{-3t}+\e^{-6t}),&
 F(t)&=\int_0^1\ell(tx)\,dx,\label{eq:F}\\
 \psi(t)&=-\frac\pi{18}+\frac13\arctan
       \frac{\sqrt3\e^{-3t}}{2+\e^{-3t}},\label{eq:psi}\\
 H(t)&=\frac{t(\e^{-6t}-1)}{12(1+\e^{-3t}+\e^{-6t})},&
 J(t)&=-\frac{t^2\psi''(t)}{27},\qquad
 W(t)=\frac{13t^3\ell'''(t)}{19440}.\label{eq:HJW}
\end{align}
Logarithms in the complex neighborhoods used below are continued
from the positive real axis. In particular,
\[
 \ii\psi(z)=-\frac{\ii\pi}{18}
 +\frac{\log(1-\bar\omega\e^{-3z})-\log(1-\omega\e^{-3z})}{6}.
\]
For $\Rea z>0$,
\begin{equation}\label{eq:dilog}
 F(z)=\frac{\pi^2/3+\Li(\e^{-9z})-3\Li(\e^{-3z})}{9z}.
\end{equation}
At zero the singularity is removable, with $F(0)=\log3$ and
$F'(0)=-3/2$. Differentiation gives
\begin{equation}\label{eq:convex}
 F''(t)=9\int_0^1x^2\frac{a+4a^2+a^3}{(1+a+a^2)^2}\,dx>0,
 \qquad a=\e^{-3tx},
\end{equation}
and
$\psi'(t)=-\sqrt3\e^{-3t}/(2(1+\e^{-3t}+\e^{-6t}))<0$.

For a degree $d$ and a parameter $t\ge0$, write
\begin{equation}\label{eq:norm}
 \M_{p,n,d}(t)=\frac{2\exp(pnF(t)+dt/n)}
 {\sqrt{2\pi pn^3F''(t)}},\qquad
 g_{p,j}(t)=\cos(p\psi(t)-2\pi j/3).
\end{equation}
Cauchy's formula on $|q|=\e^{-t/n}$ reads
\begin{equation}\label{eq:cauchy}
 \frac{c_d^{(p)}(n)}{\M_{p,n,d}(t)}=
 \sqrt{\frac{pF''(t)}{8\pi}}n^{3/2}\e^{-pnF(t)}
 \int_{-\pi}^{\pi}P_n(\e^{-t/n+\ii\theta})^p
                 \e^{-\ii d\theta}\,d\theta.
\end{equation}
The arcs about $\theta=\pm2\pi/3$ are conjugate. Their leading
normalized contribution is $g_{p,j}(t)$ for $d\equiv j\pmod3$.
Since $-\pi/18<\psi(t)\le0$ on $t\ge0$, one has
\begin{equation}\label{eq:phases}
 g_{p,0}(t)\ge\cos(p\pi/18)>0.17,\qquad
 g_{p,1}(t)\le-1/2\qquad(4\le p\le8).
\end{equation}
The phase $p\psi(t)+2\pi/3$ lies in $(0,\pi)$ and decreases
strictly. Hence $g_{p,2}$ increases strictly and has a unique zero $t_p$.
Explicitly,
\begin{equation}\label{eq:tstar}
 \phi_p=\frac{(p-3)\pi}{6p},\qquad
 u_p=\frac{2\tan\phi_p}{\sqrt3-\tan\phi_p},\qquad
 t_p=-\frac13\log u_p,\qquad \psi(t_p)=-\frac\pi{6p}.
\end{equation}
The approximate saddle parameters are $0.601565$, $0.424593$,
$0.335018$, $0.278734$, and $0.239480$, in order of increasing $p$.

We now give finite algebraic definitions of the four constants in
\eqref{eq:T}. In the remainder of this section write
$f_r=F^{(r)}(t_p)$ and $\psi_r=\psi^{(r)}(t_p)$.
For indeterminates $D_1,D_0,D_{-1},v$, let
\begin{align}\label{eq:Ar}
 A_r(y;v)={}&p\frac{F^{(r+2)}(t_p)}{(r+2)!}(-\ii y)^{r+2}
 +\ii p\frac{\psi^{(r)}(t_p)}{r!}(-\ii y)^r\notag\\
 &+p\ind_{r\ge2}\frac{H^{(r-2)}(t_p)}{(r-2)!}(-\ii y)^{r-2}
 +\ii p\ind_{r\ge4}\frac{J^{(r-4)}(t_p)}{(r-4)!}(-\ii y)^{r-4}\notag\\
 &+p\ind_{r\ge6}\frac{W^{(r-6)}(t_p)}{(r-6)!}(-\ii y)^{r-6}
 -\ii y\bigl(D_1\ind_{r=1}+(D_0+3v)\ind_{r=3}
                  +D_{-1}\ind_{r=5}\bigr).
\end{align}
Terms with negative derivative indices are omitted. Define
\begin{equation}\label{eq:Br}
 B_0=1,\quad B_r=\frac1r\sum_{k=1}^r kA_kB_{r-k},\qquad
 \G_p(y^{2r})=\frac{(2r-1)!!}{(pf_2)^r},\quad
 \G_p(y^{2r+1})=0,
\end{equation}
with $(-1)!!=1$. At $v=0$, solve successively
\begin{equation}\label{eq:Dsolve}
 \Ima\G_p(B_2)=\Ima\G_p(B_4)=\Ima\G_p(B_6)=0
\end{equation}
for $D_1,D_0,D_{-1}$. Each equation is linear in its new unknown,
with coefficient $-\psi_1/f_2\ne0$. Put
\begin{equation}\label{eq:abcd}
 \alpha_p=-pf_1/3,\qquad \beta_p=D_1/3,\qquad
 \gamma_p=(D_0-2)/3,\qquad\delta_p=D_{-1}/3.
\end{equation}
In particular,
\begin{equation}\label{eq:D1}
 D_1=\frac{f_3}{2f_2}-\frac{\psi_2}{2\psi_1},\qquad
 \alpha_p=p\int_0^1x\frac{u_p^x+2u_p^{2x}}
                              {1+u_p^x+u_p^{2x}}\,dx.
\end{equation}
Finally set
\begin{equation}\label{eq:kappa}
 \kappa_p=-3\psi_1/f_2,\qquad
 \lambda_p=-\bigl(\Ima\G_p(B_6)|_{v=1}
                         -\Ima\G_p(B_6)|_{v=0}\bigr).
\end{equation}
The Gaussian algebra gives
\begin{equation}\label{eq:gausscancel}
 \Ima\G_p(B_2)=0,\quad
 \Ima\G_p(B_4)=\kappa_pv,\quad
 \Ima\G_p(B_6)=-\lambda_pv.
\end{equation}
The last expression is linear in $v$: the only possible quadratic
contribution comes from $A_3^2/2$, whose $v^2$ term is real.
The bounds needed later are
\begin{equation}\label{eq:kappabound}
 \kappa_p-\frac{\max(\lambda_p,0)}{301}>k_p,
 \qquad k_4=3/8,\quad k_p=2/5\ (5\le p\le8).
\end{equation}
\begingroup\small
For orientation, $(\kappa_p,\lambda_p)$ are approximately
$(0.376530,0.061278)$, $(0.406331,0.031743)$,
$(0.416996,0.011375)$, $(0.422185,-0.005950)$, and
$(0.425146,-0.021449)$. Thus \eqref{eq:kappabound} applies uniformly
without imposing a sign on $\lambda_p$.
\par\endgroup

\section{Uniform analytic estimates}\label{sec:method}
\subsection{Euler--Maclaurin expansion and analytic domains}
For $a=1,2$, put $f_a(x;z)=\log(1-\omega^a\e^{-3zx})$.
Apply the shifted Euler--Maclaurin formula to the grids
$(j+1/3)/n$ and $(j+2/3)/n$, $0\le j<n$:
\[
 \sum_{j=0}^{n-1}f((j+b)/n)=n\int_0^1 f(x)\,dx+
 \sum_{k=1}^{s}\frac{B_k(b)}{k!n^{k-1}}
       \bigl(f^{(k-1)}(1)-f^{(k-1)}(0)\bigr)+R_s.
\]
For $s\ge2$, the Fourier expansion of the periodic Bernoulli
function gives
\[
 |R_s|\le\frac{2\zeta(s)}{(2\pi)^s n^{s-1}}
                         \int_0^1|f^{(s)}(x)|\,dx;
\]
see \cite[\S\S2.10, 24.8]{DLMF}. Using $s=5$ and
\[
 \begin{gathered}
 B_1(1/3)=-1/6,\quad B_2(1/3)=-1/18,\quad B_3(1/3)=1/27,\\
 B_4(1/3)=13/810,\quad B_5(1/3)=-5/243,
 \end{gathered}
\]
we obtain
\begin{equation}\label{eq:EM}
 \log P_n(\omega\e^{-z/n})=
 nF(z)+\ii\psi(z)+H(z)/n+\ii J(z)/n^2+W(z)/n^3+\mathcal R_n(z).
\end{equation}
If $\Rea z\ge0$ and $|1-\omega^a\e^{-3zx}|\ge b$ on $0\le x\le1$,
then
\begin{equation}\label{eq:logderivative}
 |\partial_x^k f_a(x;z)|\le
 (k-1)!|3z|^k\e^{-3\Rea z\,x}/b^k\qquad(k\ge1).
\end{equation}
Indeed, the logarithmic derivative is expressed by an Eulerian
polynomial with nonnegative coefficients of sum $(k-1)!$, and
$|\e^{-3zx}|\le1$.

For $z=t-\ii s$, $0\le t\le1$, $|s|\le1/4$, one may take $b=0.97$.
With $A=3\sqrt{t^2+1/16}$, the fifth endpoint term and the remaining
integral yield
\begin{equation}\label{eq:EMcompact}
 n^4|\mathcal R_n(z)|\le C_A(t):=
 \frac{5/243}{120}\frac{24A^4}{0.97^4}
 +\frac{4(1.04)}{6.28^5}\frac{24A^5}{0.97^5}.
\end{equation}
For $z=t(1-\ii v)$, $N=n/t$, retaining the exponential in
\eqref{eq:logderivative} and bounding its integral by $1/(3t)$ gives
\begin{align}\label{eq:EMscaled}
 |\mathcal R_n(z)|&\le C_*(v,b)N^{-4},\notag\\
 C_*(v,b)&=\frac{5/243}{120}\frac{24[3\sqrt{1+v^2}]^4}{b^4}
 +\frac{4(1.04)}{6.28^5}
                  \frac{24[3\sqrt{1+v^2}]^5}{3b^5}.
\end{align}
The pairs $(|v|,b)=(0.6,0.83),(1,0.65),(3,0.35)$ are valid.
To see this, split $w=tx$ at $0.6,0.4,0.15$, respectively.
On the first part use the angular distance from $1$; on the second
use $1-\e^{-3w}$. For example, the middle pair follows from
$\sin(2\pi/3-1.2)>0.65$ and $1-\e^{-1.2}>0.65$.
Together with $\pi>3.14$ and $\zeta(5)<1.04$, these comparisons
give the constants above. In a $p$th power the exponent error in
\eqref{eq:EMcompact} or \eqref{eq:EMscaled} is multiplied by $p$.

For later Cauchy estimates, let $t\in[t_-,t_+]\subset[0,1]$,
$z_+=t_++1/2$, $\rho=\e^{-3(t_--1/2)}$, and $b_*=0.55$.
On all disks $|z-t|\le1/2$ in this parameter interval, valid bounds are
\begin{align}\label{eq:compactdisks}
 M_F&=10,&M_\psi&=2,&
 M_H&=\frac{z_+(1+\rho^2)}{12b_*^2},\notag\\
 M_J&=\frac{3\sqrt3z_+^2\rho(1+\rho^2)}{54b_*^4},&
 M_W&=\frac{13\cdot54z_+^3\rho(1+\rho)}{19440b_*^3}.
\end{align}
Some of these disks cross the imaginary axis, but the angular distance
of $\omega^a\e^{-3zx}$ from the positive real axis is at least
$2\pi/3-3/2$ whenever it could approach $1$. Thus
$|1-\omega^a\e^{-3zx}|\ge\sin(2\pi/3-3/2)>0.55$.
The factors avoid the logarithmic cut, and
\[
 |\log(1-\omega^a\e^{-3zx})|\le
 \sqrt{\max\{\log(1+\e^{3/2}),-\log0.55\}^2+\pi^2}.
\]
This proves the bounds for $F,\psi$; the rational formulas in
\eqref{eq:HJW} prove those for $H,J,W$.

\subsection{A modulus bound and two product inequalities}
\begin{lemma}\label{lem:modulus}
For $t\ge0$ and $n\ge1$,
$|P_n(\omega\e^{-t/n})|\le\e^{nF(t)}$.
\end{lemma}
\begin{proof}
For every convex $f$ on $[0,1]$,
\[
 \frac{f(1/3)+f(2/3)}2\le\int_0^1 f(x)\,dx.
\]
It suffices to check affine functions and $(x-a)_+$, $0\le a\le1$.
For the latter, the difference between the right and left sides is
$a^2/2$ for $a\le1/3$, $(a-1/2)^2/2+1/24$ for
$1/3\le a\le2/3$, and $(1-a)^2/2$ for $a\ge2/3$.
Integration against the curvature measure proves the assertion.
Apply it to the convex function $\ell(tx)$ on each of the $n$
subintervals, noting that
\[
 \log|P_n(\omega\e^{-t/n})|=
 \frac12\sum_{j=0}^{n-1}\bigl(\ell(t(j+1/3)/n)
                             +\ell(t(j+2/3)/n)\bigr).
\]
\end{proof}

Let $q=r\e^{\ii\theta}$, $0<r<1$, and
$\mathcal L=-\log|P_n(q)/P_n(r\omega)|$.
The base-product estimates of \cite[Lemmas 9.1 and 9.4]{WK} are
\begin{align}
 \mathcal L&\ge\frac13\sum_{k=1}^nr^{3k}(1-\cos3k\theta)
                           -\frac{0.8}{|1-q|},\label{eq:WK1}\\
 \mathcal L&\ge\frac{(r+r^2)(1-r^{3n})}{6(1-r^3)}-5.44,
               \qquad |1-q|<1/3.\label{eq:WK2}
\end{align}
At $r=1$ we use the limiting inequalities, or continuity of the
coefficient integral. For $P_n^p$, the logarithmic loss is $p\mathcal L$.
Lemma~\ref{lem:modulus} converts these relative bounds into absolute
bounds with normalization \eqref{eq:cauchy}.

\subsection{Finite Gaussian majorants}\label{sec:majorants}
We describe a common remainder estimate for all the saddle expansions.
Let $s=3$ or $7$ be the truncation order and let $f_2>0$ denote
the second derivative of the phase function before multiplication by
$p$. Put $h=n^{-1/2}$, or $h=N^{-1/2}$ after rescaling.
After removing the constant phase and $-pf_2y^2/2$, the exact exponent is
\[
 E=\sum_{r=1}^{16}A_r(y)h^r+\rho,
 \qquad |\rho|\le pC h^8+h^{17}R(|y|).
\]
For a disk of radius $R_0$ and cutoff $|hy|\le a<R_0$, Taylor
expansions of $F,\psi,H,J,W$ through degrees $18,16,14,12,10$ give
\begin{equation}\label{eq:Rpoly}
 R(x)=\frac{p}{1-a/R_0}
 \left(\frac{M_Fx^{19}}{R_0^{19}}+
 \frac{M_\psi x^{17}}{R_0^{17}}+
 \frac{M_Hx^{15}}{R_0^{15}}+
 \frac{M_Jx^{13}}{R_0^{13}}+
 \frac{M_Wx^{11}}{R_0^{11}}\right).
\end{equation}
Let $a_r(x)$ have nonnegative coefficients that bound in modulus the
coefficients of $A_r$ throughout the parameter interval in question.
For $P_s=\sum_{r=1}^s A_rh^r$ and $h\le h_0$, set
\begin{align}
 V_s(x)&=pC h_0^{7-s}+\sum_{r=s+1}^{16}a_r(x)h_0^{r-s-1}
                         +h_0^{16-s}R(x),\label{eq:V}\\
 Q_j(x)&=\sum_{r=j}^s\binom rj a_r(x)h_0^{r-j},\qquad
 U_{s+1}(x)=[w^{s+1}]\exp\left(\sum_{j=1}^sQ_j(x)w^j\right).
 \label{eq:Bell}
\end{align}
The coefficient in \eqref{eq:Bell} is determined by recurrence
\eqref{eq:Br}. Suppose
$\Rea E,\Rea P_s(\xi,y)\le c_0+\eta y^2$ for $0\le\xi\le h$
on the major arc, and $b=pf_2/2-\eta>0$. Then
\begin{equation}\label{eq:expbound}
 \left|\e^E-\sum_{r=0}^sB_rh^r\right|
 \le h^{s+1}\e^{c_0+\eta y^2}
                   \bigl(V_s(|y|)+U_{s+1}(|y|)\bigr).
\end{equation}
Indeed, integration along the segment joining $A$ and $B$ gives
$|\e^A-\e^B|\le|A-B|\e^{\max(\Rea A,\Rea B)}$.
Apply this to $E,P_s$, and then apply Taylor's formula to
$\e^{P_s}$ as a function of $h$. Its derivatives are bounded by
the Bell polynomial in \eqref{eq:Bell}. Controlling real parts here
avoids treating the imaginary cubic term as exponential growth.

The normalized major-arc remainder is consequently bounded by
\begin{equation}\label{eq:majorerror}
 h^{s+1}\sqrt{\frac{pf_2}{2\pi}}\e^{c_0}
 \int_{\mathbb R}\bigl(V_s(|y|)+U_{s+1}(|y|)\bigr)
                                      \e^{-by^2}\,dy.
\end{equation}
Every integral is a finite sum of moments
\begin{equation}\label{eq:moments}
 \int_{\mathbb R}|y|^k\e^{-by^2}\,dy
 =\Gamma((k+1)/2)b^{-(k+1)/2}.
\end{equation}
Odd $B_r$ are odd polynomials and integrate to zero on a symmetric
interval. To extend the remaining Gaussian polynomial integrals to
the real line, use, on $|y|\ge a/h$,
$h^{r-s-1}\le(|y|/a)^{s+1-r}$. If
$I_k(Y)=\int_Y^\infty y^k\e^{-by^2}\,dy$, then
\begin{equation}\label{eq:tailsmom}
 I_0(Y)\le\frac{\e^{-bY^2}}{2bY},\quad
 I_1(Y)=\frac{\e^{-bY^2}}{2b},\quad
 I_k(Y)=\frac{Y^{k-1}\e^{-bY^2}}{2b}+\frac{k-1}{2b}I_{k-2}(Y).
\end{equation}
For precision, let $\overline B_r$ be obtained from the $a_r$ by
the positive recurrence \eqref{eq:Br}, and put
$Z_s(x)=\sum_{0\le r\le s,\ r\text{ even}}
a^{r-s-1}x^{s+1-r}\overline B_r(x)$.
An upper bound for the extension error, divided by $h^{s+1}$, is
$2\sqrt{pf_2/(2\pi)}\sum_k[x^k]Z_s(x)I_k(a/h_0)$, with
$b=pf_2/2$ in these tail moments. This specifies both parts of the
remainder entirely in terms of finite polynomials.

For the compact range, bounds on the real exponent are obtained as
follows. Majorize the absolute values of the real parts of the coefficients of $A_r$ by
$b_{rk}\ge0$. A term $b_{rk}|y|^kh^r$, $k\ge2$, is bounded by
$b_{rk}h_0^{r-k+2}a^{k-2}y^2$; the terms with $k=0,1$ give
$c_0+c_1|y|$. Apply the same estimate to the Cauchy remainder,
add $pCh_0^8$, and use
\begin{equation}\label{eq:linearabsorb}
 c_1|y|\le\epsilon y^2+c_1^2/(4\epsilon),\qquad
 \epsilon=p\inf F''/200.
\end{equation}
This bounds the exact exponent and every intervening $P_s(\xi,y)$.
It also shows why changes to the purely imaginary degree term do not
affect the real-part bound.

\section{The compact range of moving saddles}\label{sec:A}
The first correction can be made real by a degree shift independent
of the power. Define
\begin{equation}\label{eq:moving}
 D(t)=\frac{F'''(t)}{2F''(t)}-\frac{\psi''(t)}{2\psi'(t)},\qquad
 d_{A,p}(t)=-pn^2F'(t)+nD(t),\quad 0\le t\le1.
\end{equation}
The function $D$ extends to zero with $D(0)=0$.
We shall use the critical saddle intervals
\begin{equation}\label{eq:criticalintervals}
 I_4=[0.55,0.65],\qquad I_p=[0.2,0.5]\quad(5\le p\le8).
\end{equation}
They contain $t_p$ in their interiors.

\begin{proposition}\label{prop:A}
For $n\ge301$, $d_{A,p}$ decreases strictly from $3pn^2/2$ to
$d_{A,p}(1)$. If $d=d_{A,p}(t)$ is an integer and $d\equiv j\pmod3$,
then
\begin{equation}\label{eq:Aexp}
 \frac{c_d^{(p)}(n)}{\M_{p,n,d}(t)}
 =g_{p,j}(t)(1+R_{1,p}(t)/n)+\epsilon,
 \qquad |\epsilon|<E_p/n^2.
\end{equation}
The following bounds hold; the last three columns apply on $I_p$.
\begin{center}
\begin{tabular}{crrrrr}
\toprule
$p$&\shortstack{$R_{1,p}$\\lower bound}&$E_p$&\shortstack{improved\\$E_p$}
 &\shortstack{$g'_{p,2}$\\lower bound}&\shortstack{$|((d_{A,p}-2)/3)'|/n^2$\\upper bound}\\\midrule
4&$-3$&300&15&0.42&2\\
5--8&$-5$&50&12&0.6&5\\\bottomrule
\end{tabular}
\end{center}
\end{proposition}

\subsection{The first correction and the major-arc remainder}
At a moving saddle, use \eqref{eq:Ar} with derivatives at $t$,
the sole degree correction $D_1=D(t)$, and no $D_0,D_{-1},v$ terms.
Replace $\G_p$ by the Gaussian functional with moments
$(2r-1)!!/(pF''(t))^r$. The imaginary part of its second coefficient is
\begin{equation}\label{eq:phasecorrection}
 -\frac{\psi''}{2F''}+\frac{F'''\psi'}{2(F'')^2}
                         -\frac{\psi'D}{F''}=0.
\end{equation}
Consequently $R_{1,p}=\G_{p,t}(B_2)$ is real. Writing
$f_j=F^{(j)}(t)$ and $\psi_j=\psi^{(j)}(t)$, it is explicitly
\begin{equation}\label{eq:R1explicit}
 R_{1,p}=pH+\frac{p\psi_1^2}{2f_2}
 +\frac{f_4}{8pf_2^2}-\frac{5f_3^2}{24pf_2^3}
 +\frac{f_3D}{2pf_2^2}-\frac{D^2}{2pf_2}.
\end{equation}
This establishes the multiplicative form of the first correction
in \eqref{eq:Aexp}.

For the major-arc error, apply Section~\ref{sec:majorants} with
$s=3$, $h_0=301^{-1/2}$, $a=1/4$, $R_0=1/2$,
$C=C_A(t)$, and the disk bounds \eqref{eq:compactdisks}.
The construction \eqref{eq:linearabsorb} provides $c_0,\eta$ and a
positive $b$. Equations \eqref{eq:majorerror} and \eqref{eq:tailsmom}
then bound the error after the $1/n$ correction, including extension
of the Gaussian polynomial integrals.

\subsection{Complementary arcs in the compact range}\label{sec:compacttails}
We give the full contour estimates, also needed at the fixed critical
saddles. Put $a=1/4$, $r=\e^{-t/n}$, and consider first
$\theta=\pm2\pi/3+s/n$ with $a\le|s|\le1$.
When $t$ is small, the right Riemann sum of $1-\cos(3sx)$ is at
least its integral. Equation~\eqref{eq:WK1} therefore supplies the
logarithmic gap per $n$
\begin{equation}\label{eq:smallgap}
 \mu(t)=\e^{-3t}\bigl(1-\sin(3/4)/(3/4)\bigr)/3.
\end{equation}
The denominator on these arcs exceeds $1.7$ in the ranges where
this estimate is used below.

For the rest of the compact range we use
\begin{equation}\label{eq:compactgap}
 F(t)-\Rea F(t-\ii s)\ge
       F''(t)\sin^2(3/4)/18\qquad(1/4\le|s|\le1).
\end{equation}
Here is a direct proof. For $a_x=\e^{-3tx}$,
\begin{align*}
 |1+a_x\e^{\ii\phi}+a_x^2\e^{2\ii\phi}|^2
 ={}&(1+a_x+a_x^2)^2
 -2a_x(1+a_x^2)(1-\cos\phi)\\
 &-2a_x^2(1-\cos2\phi).
\end{align*}
Apply $-\log(1-X)\ge X$. For $1/4\le s\le1/2$, retain both terms
and use
\[
 \sin(3sx/2)\ge x\sin(3s/2)\ge\tfrac{x}{2}\sin(3/4),\qquad
 \sin(3sx)\ge x\sin(3s)\ge x\sin(3/4).
\]
The resulting integrand is at least
$\tfrac12x^2\sin^2(3/4)
(a_x+4a_x^2+a_x^3)/(1+a_x+a_x^2)^2$.
For $1/2\le s\le1$, retain the first term and use
$\sin(3sx/2)\ge x\sin(3/4)$ and
$a_x(1+a_x^2)/(a_x+4a_x^2+a_x^3)\ge1/3$.
Integration and \eqref{eq:convex} prove \eqref{eq:compactgap};
negative $s$ follows by conjugation.

To pass from this gap to the product, a second-order
Euler--Maclaurin estimate is sufficient. The two choices used are
\begin{center}
\begin{tabular}{crrr}
\toprule
$p$&\shortstack{lower endpoint\\$t_0$}&\shortstack{logarithmic\\denominator bound}
 &\shortstack{additional real exponent of $P_n^p$\\upper bound}\\\midrule
4&0.35&0.39&1\\
5--8&0.2&0.25&$p/2$\\\bottomrule
\end{tabular}
\end{center}
For $p=4$ the bound follows from
\[
 4\left[\frac16\log\frac{1+\e^{-1.05}}{1-\e^{-1.05}}
 +\frac{\sqrt2(1+\e^{-2.10})}{12\cdot301(1-\e^{-1.05})^2}
 +\frac{18}{6\cdot301\cdot0.39^2}\right]<1.
\]
The denominator bound follows by splitting $x$ at $1/2$ and using
the angular separation on the first part and
$1-\e^{-0.525}>0.39$ on the second.
For $5\le p\le8$, the same split gives
$\sin(2\pi/3-1.5)>0.25$ and $1-\e^{-0.3}>0.25$; the base-product
additional exponent is bounded by
\[
 \frac16\log\frac{1+\e^{-0.6}}{1-\e^{-0.6}}
 +\frac{\sqrt2(1+\e^{-1.2})}{12\cdot301(1-\e^{-0.6})^2}
 +\frac{18}{6\cdot301\cdot0.25^2}<1/2.
\]

We now state explicit envelopes for all remaining arcs. Let
$A=3t$, $H_A=(1-\e^{-A})/A$, with $H_0=1$, and put
\begin{equation}\label{eq:compactgaps}
 \begin{split}
 L&=H_A\e^{-A/301},\qquad
 U_A=\frac{1+\e^{-A}}
 {\sqrt{(A\e^{-A/301})^2+4\e^{-A/301}[301\sin(3/602)]^2}},\\
 \Delta&=L-U_A,\qquad
 P_+=\frac{(\e^{-t/301}+\e^{-2t/301})H_A}{6},\qquad
 \eta_p(t)=\sqrt{pF''(t)/(8\pi)}.
 \end{split}
\end{equation}
The geometric sum gives a weighted cosine gap of at least $n\Delta$
at distance at least $1/n$ from every cube root of unity.
Using \eqref{eq:WK1} on $|1-q|\ge1/3$ and \eqref{eq:WK2}
on $|1-q|<1/3$, the normalized complementary-arc error multiplied
by $n^r$ is bounded by the sum
\begin{align}\label{eq:compacttails}
 \mathcal T_{p,r}(t;n)={}&
 4(1-a)\eta_p n^{r+1/2}\e^{b_0-pn\mu}\notag\\
 &+2\pi\eta_p n^{r+3/2}\e^{1+2.4p-pn\Delta/3}
 +2\pi\eta_p n^{r+3/2}\e^{1+5.44p-pnP_+}.
\end{align}
For $t<t_0$, take \eqref{eq:smallgap} and
$b_0=1+0.8p/1.7$. For $t\ge t_0$, take
$\mu=F''(t)\sin^2(3/4)/18$ and $b_0=1$ if $p=4$,
$b_0=p/2$ otherwise. At $t=t_0$ either applicable bound may be used.
The extra factors $\e$ in the far-arc envelopes are harmless.
The neighborhood $|\theta|<1/n$ of the positive axis lies in
$|1-q|<1/3$, so the indicated regions cover the complementary contour.
Each envelope is decreasing for $n\ge301$ whenever its decay
constant times $301$ exceeds the corresponding exponent of $n$.
These strict inequalities hold for $r=2$ over the compact range;
they also hold for $r=4$ at the critical saddles with $p\ge5$.

\subsection{Uniform bounds and the sign away from the transition}
Subdivide $[0,1]$ into the 1000 closed intervals
$[i/1000,(i+1)/1000]$. On each interval use the analytic series in
Appendix~\ref{app:series}, the positive majorants above, and
\eqref{eq:compacttails} at $n=301$. The resulting total constants
in \eqref{eq:Aexp} are bounded as follows:
\begin{center}
\begin{tabular}{crrr}
\toprule
$p$&entire $[0,1]$&critical interval $I_p$&lower bound for $R_{1,p}$\\\midrule
4&232.299&5.242&$-3$\\
5&17.715&5.549&$-0.774$\\
6&9.670&4.265&$-0.786$\\
7&7.601&3.512&$-0.817$\\
8&6.382&3.040&$-0.860$\\\bottomrule
\end{tabular}
\end{center}
The same substitutions give $|D'|<10$, and
$0.43<F''<2.1$ for $p=4$, while $pF''/4>1/2$ suffices for the
other powers. They give the phase and degree derivative bounds
in Proposition~\ref{prop:A} on the cells contained in $I_p$.
In particular, $d_{A,p}'=-pn^2F''+nD'<0$ for $n\ge301$, and
$d_{A,p}(0)=3pn^2/2$. This proves the proposition.

Write
\begin{equation}\label{eq:Sn}
 S_{p,n}=\frac{d_{A,p}(t_p)-2}{3}
 =\alpha_p n^2+\beta_p n-2/3.
\end{equation}
The defining bounds give $|T_{p,n}-S_{p,n}|<1$ for $n\ge301$.
If $|m-T_{p,n}|>130$ and the corresponding saddle lies in $I_p$,
then $|m-S_{p,n}|>129$. The mean value theorem and
Proposition~\ref{prop:A} give a main term with modulus at least
\[
 \frac{0.99\cdot0.42\cdot129}{2n^2}
 =\frac{26.8191}{n^2}>\frac{15}{n^2}\quad(p=4),
\]
or
\[
 \frac{0.98\cdot0.6\cdot129}{5n^2}
 =\frac{15.1704}{n^2}>\frac{12}{n^2}\quad(5\le p\le8).
\]
Its sign is $\sgn(T_{p,n}-m)$. Outside $I_p$, monotonicity and
the endpoint phase values give $|g_{4,2}|>0.02$ or
$|g_{p,2}|>0.05$ for $p\ge5$. These dominate the respective errors
$300/n^2$ and $50/n^2$ after multiplication by the positive amplitude.
Finally, \eqref{eq:phases} shows that every residue-$1$ coefficient
in the compact range is negative and, for $p\ge5$, every residue-$0$
coefficient is positive.

\section{The critical strip}\label{sec:local}
\begin{proposition}\label{prop:local}
Let $n\ge301$ and $v=m-T_{p,n}$. For $|v|\le1$,
\begin{equation}\label{eq:local}
 \left|\frac{z_{p,n,m}}{\M_{p,n,3m+2}(t_p)}
       +\frac{\kappa_p-\lambda_p/n}{n^2}v\right|
 <\frac{C_p}{n^4},\qquad
 C_4=560,\quad C_p=60\ (5\le p\le8).
\end{equation}
Moreover, \eqref{eq:signplus}--\eqref{eq:signminus} hold throughout
$|m-T_{p,n}|\le130$.
\end{proposition}
\begin{proof}
Write $f_r=F^{(r)}(t_p)$.
Use $h=n^{-1/2}$ and $z=t_p-\ii hy$ on $|hy|\le1/4$.
The degree is
$d=-pf_1n^2+D_1n+D_0+D_{-1}/n+3v$, and substitution in
\eqref{eq:EM} gives exactly \eqref{eq:Ar}. The constant phase is
$\e^{\ii p\psi(t_p)-4\pi\ii/3}=\ii$.
The two conjugate arcs in \eqref{eq:cauchy}, after the change
$y=n^{3/2}(\theta-2\pi/3)$, give Gaussian normalization
$\sqrt{pf_2/(2\pi)}$. Hence their polynomial contribution through
order $h^7$ is, by \eqref{eq:gausscancel},
\begin{equation}\label{eq:localpolynomial}
 \Rea\left(\ii\sum_{r=0}^7\G_p(B_r)h^r\right)
 =-\kappa_pv/n^2+\lambda_pv/n^3.
\end{equation}

For $p\ge5$, use $s=7$ in Section~\ref{sec:majorants}, the disk
bounds \eqref{eq:compactdisks} at $t=t_p$, and $C=C_A(t_p)$.
The real-part construction \eqref{eq:linearabsorb} applies uniformly
for $|v|\le130$, since $v$ enters only a purely imaginary term.
The complementary-arc error multiplied by $n^4$ is bounded by
$\mathcal T_{p,4}(t_p;301)$ in \eqref{eq:compacttails}.
Together, \eqref{eq:majorerror}, \eqref{eq:tailsmom}, and that envelope
give the strict upper bounds $54.652$, $19.167$, $12.408$, and
$8.599$ for $p=5,6,7,8$, respectively, when $|v|\le1$.

We give the sharper estimates required for $p=4$ in detail.
On $|\Ima z|\le1/4$ at $\Rea z=t_4$, the logarithmic denominators
exceed $0.97$, $|3z|<1.96$, and
$n^4|\mathcal R_n(z)|<0.412894<1/2$.
On $|z-t_4|\le1/2$, the denominators exceed $0.55$ and the
Cauchy bounds for $(F,\psi,H,J,W)$ can be sharpened to
$(5,1,1/2,2,1/2)$. For example,
$\psi''(z)=(3\sqrt3/2)u(1-u^2)/(1+u+u^2)^2$ with
$u=\e^{-3z}$, together with \eqref{eq:HJW}, gives the latter bounds.
Thus \eqref{eq:Rpoly} becomes
\[
 R(x)=40\,2^{19}x^{19}+8\,2^{17}x^{17}+4\,2^{15}x^{15}
                  +16\,2^{13}x^{13}+4\,2^{11}x^{11}.
\]
The real-coefficient majorants give
$\Rea E,\Rea P_7\le0.03+0.05y^2$ and
$2F''(t_4)-0.05>1.85$. With
$\sqrt{2F''(t_4)/\pi}<0.78$, the two integrals in
\eqref{eq:majorerror} are bounded, after multiplication by $n^4$,
by $2.685580$ and $555.780836$. The extension error from
\eqref{eq:tailsmom}, using the smaller Gaussian exponent $1.9$,
is less than $0.000051$.

For the middle arcs, a second-order Euler--Maclaurin estimate gives
\[
 |P_n(\omega\e^{-(t_4-\ii s)/n})|^4
 \le\exp(4n\Rea F(t_4-\ii s)+1)\qquad(|s|\le1).
\]
Indeed, splitting $x$ at $19/40$ gives a denominator bound $0.57$;
the first endpoint contribution has real part less than $0.057$,
and the second endpoint plus remainder has modulus less than $8/n$.
The dilogarithm formula gives the stronger gap
\begin{equation}\label{eq:localgap}
 F(t_4)-\Rea F(t_4-\ii s)>0.3s^2\qquad(1/4\le|s|\le1).
\end{equation}
To obtain this bound, cover $[1/4,1]$ by the 750 closed intervals
$[k/1000,(k+1)/1000]$, $250\le k<1000$, and retain 60 terms
in each dilogarithm series. The tail estimate in
Appendix~\ref{app:series} gives, on each interval, a lower bound
greater than $0.010024$ for the left side minus the right side
of \eqref{eq:localgap}. Negative $s$ follows by conjugation.
The normalized middle-arc error, multiplied by $n^4$, is therefore
at most
\[
 \frac{1.56\e}{2.4(1/4)}n^{7/2}\e^{-1.2n/16}<0.524872.
\]
For the remaining arcs use \eqref{eq:compactgaps} at $t=t_4$.
The geometric gap satisfies $\Delta>0.12$, and the positive-axis
main term exceeds $0.15n$. Thus the errors multiplied by $n^4$
are bounded by
\[
 2.5n^{11/2}\e^{10.6-0.16n}<0.005224,
 \qquad
 2.5n^{11/2}\e^{22.76-0.6n}<4\cdot10^{-55},
\]
on $|1-q|\ge1/3$ and $|1-q|<1/3$, respectively.
All these envelopes decrease for $n\ge301$. Their regions, together
with the middle arcs, cover the complement of the major arcs.
The sum of the unrounded bounds is less than $558.996561<560$.
This proves \eqref{eq:local} for every power.

The same polynomial bounds apply to larger intervals $|v|\le B$.
Table~\ref{tab:windows} lists upper bounds $C_{p,B}$ for the error
constant multiplying $n^{-4}$. For a given row, the remaining
possible sign ambiguity is restricted to
$|v|<C_{p,B}/(k_pn^2)$ by \eqref{eq:kappabound}.
\begin{table}[ht]
\centering
\begin{tabular}{crrr}
\toprule
$p$&$B$&$C_{p,B}$&inner radius at $n=301$\\\midrule
4&1&560&$<1500/301^2$\\
4&20&17000&$<0.501$\\
4&60&340000&$<10.008$\\
4&100&1700000&$<50.037$\\
4&120&3100000&$<91.243$\\
4&130&4060000&$<119.499$\\\midrule
5--8&1&60&$\le150/301^2$\\
5--8&20&5000&$<0.138$\\
5--8&60&120000&$<3.312$\\
5--8&130&1400000&$<38.631$\\\bottomrule
\end{tabular}
\caption{Overlapping windows at the critical saddle.}\label{tab:windows}
\end{table}
Each inner radius is smaller than the preceding half-width in its
group, and all radii decrease with $n$. Starting with $B=1$ and
successively using the larger windows proves the sign conclusions
throughout $|m-T_{p,n}|\le130$.
\end{proof}

\section{The intermediate scaled range}\label{sec:B}
Use the ordinary saddle relation
\begin{equation}\label{eq:ordinary}
 d_{B,p}(t)=-pn^2F'(t),\qquad 0.99\le t\le3,
 \qquad N=n/t,\qquad G_t(w)=tF(tw).
\end{equation}
The five functions in the scaled expansion are
\begin{equation}\label{eq:scaledfunctions}
 G_t(w),\quad\psi(tw),\quad H(tw)/t,\quad
 J(tw)/t^2,\quad W(tw)/t^3.
\end{equation}
The normalization is unchanged because
$n^3F''(t)=N^3G_t''(1)$.

\begin{proposition}\label{prop:B}
Suppose $n\ge301$, $0.99\le t\le3$, and $d=d_{B,p}(t)$ is an
integer. Then $c_d^{(p)}(n)/\M_{p,n,d}(t)>0.1$ for
$d\equiv2\pmod3$, and this ratio is less than $-0.1$ for
$d\equiv1\pmod3$. For $5\le p\le8$, the ratio is also greater
than $0.1$ when $d\equiv0\pmod3$.
\end{proposition}
\begin{proof}
Use Section~\ref{sec:majorants} with $h=N^{-1/2}$, $s=3$,
cutoff $a=0.4$, and disk radius $R_0=0.8$ about $w=1$.
In \eqref{eq:Ar}, replace the five functions by
\eqref{eq:scaledfunctions}, take their derivatives at $w=1$,
and omit the degree corrections. The Gaussian moments are
$(2r-1)!!/(pG_t''(1))^r$. The first correction $C_{1,p}$ is the
Gaussian average of $B_2$ and may be complex. The major-arc
approximation, including its Gaussian extension, therefore has error
at most $|C_{1,p}|/N+E_{B,p}/N^2$ from its leading phase.

We specify the bounds giving $E_{B,p}$. On a parameter interval
$[t_-,t_+]$, set $\rho=\e^{-3t_-(1-R_0)}$, $b_*=1-\rho$,
and $z_+=1+R_0$. Bounds for the functions in
\eqref{eq:scaledfunctions} are
\begin{align}\label{eq:scaleddisks}
 M_G&=\frac{\pi^2/27+(\rho^3/(1-\rho^3)+3\rho/(1-\rho))/9}
                      {1-R_0},\notag\\
 M_\psi&=\pi/18-\log(1-\rho)/3,&
 M_H&=\frac{z_+(1+\rho^2)}{12b_*^2},\notag\\
 M_J&=\frac{3\sqrt3z_+^2\rho(1+\rho^2)}{54b_*^4},&
 M_W&=\frac{13\cdot54z_+^3\rho(1+\rho)}{19440b_*^3}.
\end{align}
These follow directly from \eqref{eq:dilog} and \eqref{eq:HJW}.
For the real exponent, a fourth derivative bound on $\Rea w=1$
is more efficient. Put $c=\pi^2/27$, $u=\e^{-3t}$, and
$S_k(u)=\sum_{l\ge1}l^ku^l$. Then
\begin{equation}\label{eq:F4}
 \frac{|G_t^{(4)}(w)|}{4!}\le
 c+\frac19\sum_{r=0}^4
 \frac{(9t)^rS_{r-2}(u^3)+3(3t)^rS_{r-2}(u)}{r!}=:L_4.
\end{equation}
The functions $S_{-2},S_{-1},S_0,S_1,S_2$ are, respectively,
$\Li(u)$, $-\log(1-u)$, $u/(1-u)$, $u/(1-u)^2$, and
$u(1+u)/(1-u)^3$. The Gaussian-removed real part contributed by
$pNG_t$ is at most $pL_4a^2y^2$.

Let $h_0=N_0^{-1/2}$ and
$\widetilde\psi_j=t^j\psi^{(j)}(t)/j!$. The phase term is bounded
by $c_1|y|$, where
\[
 c_1=h_0\max\left\{\frac{pt\sqrt3u}{2(1-u)^2},
                 p|\widetilde\psi_1|+p|\widetilde\psi_3|a^2\right\}.
\]
For use on vertical segments, denote by $H_v,J_v,W_v$ the rational
bounds in \eqref{eq:scaleddisks} with
$z_+=\sqrt{1+v^2}$, $\rho=u$, and $b_*=1-u$.
The remaining real exponent is bounded by
\[
 c_0=p\left(H_a/N_0+J_a/N_0^2+W_a/N_0^3
                         +C_*(a,0.83)/N_0^4\right).
\]
Take $\epsilon=p\inf G_t''(1)/400$,
$\eta=p\sup L_4a^2+\epsilon$, and replace $c_0$ by
$c_0+c_1^2/(4\epsilon)$. These bounds apply to the exact exponent
and to every $P_3(\xi,y)$. The resulting
$b=p\inf G_t''(1)/2-\eta$ is positive on the intervals used below.
Equations \eqref{eq:V}--\eqref{eq:tailsmom}, with
$C=C_*(a,0.83)$ and disk bounds \eqref{eq:scaleddisks}, now give
$E_{B,p}$.

Here are explicit envelopes for the complementary arcs. In the
following expressions use $u=\e^{-3t_-}$ and
$\eta_p=\sqrt{p\sup G_t''(1)/(8\pi)}$. Put
\begin{equation}\label{eq:scaledgaps}
 \begin{split}
 R&=\frac{u^3/(1-u^3)+3u/(1-u)}9,\qquad
 \Delta_m=\frac{ca^2}{1+a^2}-2R,\\
 L&=\frac{1-u}{N_0(\e^{3/N_0}-1)},\qquad
 \Delta_b=L-\frac{1+u}{\sqrt{9+4[N_0\sin(3b/(2N_0))]^2}},\\
 P_+&=\frac{(\e^{-1/N_0}+\e^{-2/N_0})(1-u)}{18}.
 \end{split}
\end{equation}
The dilogarithm representation gives
$G_t(1)-\Rea G_t(1-\ii v)\ge cv^2/(1+v^2)-2R$;
thus $\Delta_m$ applies for $a\le|v|\le3$.
An upper bound for the additional real exponent is
\begin{equation}\label{eq:Am}
 A_{m,p}=p\left[\frac16\log\frac{1+u}{1-u}
       +\frac{H_3}{N_0}+\frac{J_3}{N_0^2}
       +\frac{W_3}{N_0^3}+\frac{C_*(3,0.35)}{N_0^4}\right].
\end{equation}
The total normalized complementary-arc bound is
\begin{align}\label{eq:Btails}
 E_{{\rm tail},p}={}&
 4(3-a)\eta_p\sqrt{N_0}\e^{A_{m,p}-pN_0\Delta_m}
 +48\eta_p\sqrt{N_0}\e^{8p/15-(p/3)N_0\Delta_3}\notag\\
 &+2\pi\eta_pN_0^{3/2}\e^{12p/5-(p/3)N_0\Delta_{15}}
 +2\pi\eta_pN_0^{3/2}\e^{136p/25-pN_0P_+}.
\end{align}
The second term covers distances $3/N$ through $15/N$ from the two
nonreal roots; here $|1-q|>1.5$. The third uses \eqref{eq:WK1}
away from all three roots on $|1-q|\ge1/3$. The last uses
\eqref{eq:WK2} and covers the remaining positive-axis neighborhood,
since $(15+1)/N_0<1/3$. The geometric sums giving \eqref{eq:scaledgaps},
and monotonicity of $N\sin(b/N)$, make these bounds uniform for
$N\ge N_0$. All envelopes decrease because their decay constants
times $N_0$ exceed their exponents of $N$.

Use the 402 closed intervals $[i/200,(i+1)/200]$,
$198\le i<600$, and $N_0=301\cdot200/(i+1)$ on each.
Subtract
\[
 |C_{1,p}|/N_0+E_{B,p}/N_0^2+E_{{\rm tail},p}
\]
from $g_{p,0}$, $-g_{p,1}$, and $g_{p,2}$, as appropriate.
The resulting strict lower bounds are listed in
Table~\ref{tab:B}; all gaps and positivity conditions just used hold
on every interval. This proves the proposition.
\begin{table}[ht]
\centering
\begin{tabular}{crrr}
\toprule
$p$&residue $0$&negative of residue $1$&residue $2$\\\midrule
4&---&0.9036249&0.1017609\\
5&0.6276837&0.9599507&0.2635920\\
6&0.4879332&0.9877481&0.4158374\\
7&0.3312488&0.9738453&0.5561312\\
8&0.1633406&0.9292018&0.6814778\\\bottomrule
\end{tabular}
\caption{Strict lower bounds for the normalized sign margins in the
intermediate range.}\label{tab:B}
\end{table}
\end{proof}

The same explicit bounds give the overlap in degree:
\begin{equation}\label{eq:ABoverlap}
 d_{B,p}(0.99)-d_{A,p}(1)>0.016n^2\qquad(n\ge301).
\end{equation}
For $p\ge5$ the stronger bound $0.02n^2$ holds. They also give
\begin{equation}\label{eq:separation}
 301^2\bigl(\alpha_p+pF'(0.99)/3\bigr)>131,
 \quad 0<\beta_p<1,
 \quad -1<\gamma_p<0,
 \quad 0<\delta_p<1.
\end{equation}
Consequently every residue-$2$ degree in the intermediate range has
$m<T_{p,n}-130$. For the quartic case, the simpler bounds
$-4F'(0.99)/3<0.41$ and $\alpha_4>0.74$ already imply this separation.

\section{The unbounded edge range}\label{sec:edge}
We now treat $t\ge3$ and $N=n/t\ge N_0=124/5$ with the ordinary
saddle relation. Throughout this section put
\[
 u=\e^{-3t}\le U=1/8000,\qquad
 r=\e^{-1/N}>24/25,\qquad c=\pi^2/27.
\]
The estimates will be uniform in the whole unbounded range of $t$.

\subsection{Weighted product bounds}
Retaining the weight in the proof of \cite[Lemma 9.1]{WK} gives
\begin{equation}\label{eq:weighted}
 \mathcal L\ge\sum_{k=1}^n f(3k/N)(1-\cos3k\theta)
                   -\frac{M(u)}{|1-r\e^{\ii\theta}|},\qquad
 f(x)=\frac{\e^{-x}}{1+\e^{-x}+\e^{-2x}},
\end{equation}
where
\[
 M(u)=\frac2{\sqrt3}\left(\frac\pi3-
       \arctan\frac{1+2u}{\sqrt3}\right)+\frac{2u}{1+u+u^2}
 <M_0:=121/200.
\]
Indeed, $M(u)\le\pi/(3\sqrt3)+2U<M_0$.
The function $f$ decreases and $f(0)=1/3$. Right Riemann sums
and Abel summation yield
\begin{equation}\label{eq:weighted-sums}
 \sum_{k=1}^n f(3k/N)\ge NA_0-1/3,\qquad
 \left|\sum_{k=1}^n f(3k/N)\e^{3\ii k\theta}\right|
 \le\frac1{3|\sin(3\theta/2)|},\quad
 A_0=\frac\pi{9\sqrt3}-\frac U3.
\end{equation}
Here $\int_0^\infty f(x)\,dx=\pi/(3\sqrt3)$ and its omitted tail
after $3t$ is at most $u$. At angular distance at least $B/N$
from every cube root, it follows that
\begin{equation}\label{eq:weighted-gap}
 \mathcal L\ge NG_B-1/3-M_0/|1-q|,\qquad
 G_B=A_0-\frac1{3N_0\sin(3B/(2N_0))}.
\end{equation}
We shall use $B=3,7,15$. Monotonicity of $N\sin(3B/(2N))$
justifies the common lower bound $N_0$ in this formula.

\begin{lemma}\label{lem:positive-axis}
In the edge range, $|1-q|<1/3$ implies
\begin{equation}\label{eq:newpositive}
 \mathcal L\ge NA_0-38/15.
\end{equation}
\end{lemma}
\begin{proof}
Set $\rho=|1-q|$, $h=\min(\lfloor1/(3\rho)\rfloor,n)$,
$x=r^{3h}$, $L=|1+q+q^2|$, and $D_u=1+u+u^2$.
Let $S_{\rm tail}$ be the sum of $r^k/(1+r^k+r^{2k})$ over
$3h<k\le3n$, $3\nmid k$, and let $S_{\rm all}$ be the same sum
over $1\le k\le3n$. The prefix/tail decomposition in
\cite[(9.6)--(9.8)]{WK}, retaining its weights, gives
\begin{equation}\label{eq:prefix}
 \mathcal L\ge
 \left(\frac12-\frac{1-r^3}{\rho L}\right)S_{\rm tail}
 +h[2+\log(1+\sqrt x+x)]-\frac{4u}{D_u\rho L}-\frac12\log3.
\end{equation}
If $h=n$, the prefix itself gives
$\mathcal L\ge2n-(\log3)/2$, which is stronger than
\eqref{eq:newpositive} because $n=tN\ge3N$.
If $h<n$, there are $2h$ prefix weights, each at most $1/3$.
Thus $S_{\rm tail}/2\ge S_{\rm all}/2-h/3$, and
\[
 (1-r^3)S_{\rm tail}\le(r+r^2)(x-u)\le2(x-u).
\]
The condition $\rho<1/3$ implies
\[
 \cos\theta>r/2+4/(9r),\qquad
 L\ge r+(1+r^2)\cos\theta>11/4\quad(r\ge24/25).
\]
Using also $1/(3\rho)<h+1$ in \eqref{eq:prefix}, we obtain
\[
 \mathcal L\ge\tfrac12S_{\rm all}
 +h[5/3+\log(1+\sqrt x+x)]
 -\frac{24}{11}(h+1)(x+u)-\tfrac12\log3.
\]
The concavity of $\log(1+\sqrt x+x)$ on $[0,1]$ gives the lower
bound $x\log3$. The resulting coefficient of $h$ is positive.
Since $h\ge1$, $x\le1$, and $u\le U$, the last three terms
are therefore bounded below by
\[
 \frac53+\frac12\log3-\frac{48}{11}(1+U)>-11/5.
\]
Finally,
$S_{\rm all}/2\ge\sum_{k=1}^nf(3k/N)\ge NA_0-1/3$,
which proves the claim.
\end{proof}

\subsection{A positive real major-arc integral}
Write $G_t(w)=c/w+R_t(w)$. Termwise differentiation of the
dilogarithm series gives, for $\Rea w=1$ and $t\ge3$,
\begin{equation}\label{eq:E2}
 |R_t''(w)|/2\le E_2:=0.0021.
\end{equation}
Indeed, a uniform upper bound is
\[
 \frac19\left[\frac{\e^{-27}}{1-\e^{-27}}(1+27+27^2/2)
 +\frac{3\e^{-9}}{1-\e^{-9}}(1+9+9^2/2)\right]
 <0.002077655.
\]
The bound is uniform because $t^r\e^{-3lt}$ decreases for
$t\ge3$, $r\le2$, $l\ge1$. In particular, on writing
$f_2=G_t''(1)$ in this subsection,
\begin{equation}\label{eq:edgecurvature}
 f_{2,-}:=2c-2E_2\le f_2\le2c+2E_2=:f_{2,+}.
\end{equation}

Set $\theta=2\pi/3+v/N$, $y=v\sqrt N$, and $|v|\le1$.
Consider the exponent
\begin{equation}\label{eq:Kt}
 K_{p,t}(y)=pN[G_t(1-\ii v)-G_t(1)+\ii vG_t'(1)]
                            -\frac{p(1-\ii v)}{12N}.
\end{equation}
Its symmetric integral is real. For $a_p=pc$ and $\beta_p^*=p/12$,
the reference exponent is
\begin{equation}\label{eq:Kinfty}
 K_{p,\infty}(y)=-\frac{a_py^2}{1+y^2/N}-\frac{\beta_p^*}{N}
 +\ii\left(-\frac{a_py^3/\sqrt N}{1+y^2/N}
                         +\frac{\beta_p^*y}{N^{3/2}}\right).
\end{equation}
Equation~\eqref{eq:E2} gives
$|K_{p,t}-K_{p,\infty}|\le pE_2y^2$.

Choose
\begin{equation}\label{eq:Yp}
 Y_p=\frac85(4/p)^{1/3}.
\end{equation}
On $|y|\le Y_p$, the modulus of the reference phase is at most
$a_pY_p^3/\sqrt{N_0}+\beta_p^*Y_p/N_0^{3/2}<\sqrt2$.
We may therefore use $\cos\Phi\ge1-\Phi^2/2\ge0$, together with
$\e^{-a_py^2/(1+y^2/N)}\ge\e^{-a_py^2}$.
The full Gaussian moments of
$\tfrac12(a_p|y|^3/\sqrt{N_0}
                   +\beta_p^*|y|/N_0^{3/2})^2$
give the three losses below. On $Y_p<|y|\le\sqrt N$, a possible
negative contribution is bounded by $\e^{-a_py^2/2}$.
It follows that
\begin{align}\label{eq:J0}
 \sqrt{a_p/\pi}\int_{-\sqrt N}^{\sqrt N}
             \e^{K_{p,\infty}(y)}\,dy&\ge J_{0,p},\notag\\
 J_{0,p}=\e^{-p/(12N_0)}\bigg[1
 &-\frac{\e^{-a_pY_p^2}}{\sqrt{\pi a_p}Y_p}
 -\frac{15}{16a_pN_0}-\frac{p}{16a_pN_0^2}
 -\frac{p^2}{576a_pN_0^3}
 -\frac{2\e^{-a_pY_p^2/2}}{\sqrt{\pi a_p}Y_p}\bigg].
\end{align}
Let
$b_p^*=a_p/2-pE_2$ and $\nu_{p,+}=\sqrt{pf_{2,+}/(2\pi)}$.
Exponential error propagation and the Gaussian second moment give
\[
 E_{F,p}=\nu_{p,+}\frac{pE_2\sqrt\pi}{2(b_p^*)^{3/2}}.
\]
Consequently the true real integral satisfies
\begin{equation}\label{eq:core}
 \sqrt{\frac{pf_2}{2\pi}}
 \int_{-\sqrt N}^{\sqrt N}\e^{K_{p,t}(y)}\,dy
 \ge\sqrt{\frac{f_{2,-}}{2c}}J_{0,p}-E_{F,p}=:Q_p>0.85.
\end{equation}
For $p=4$, $Y_4=8/5$, $J_{0,4}>0.86550207$,
$E_{F,4}<0.008291$, and $Q_4>0.85472206$.
For $p=5,6,7,8$, the respective lower bounds for $Q_p$ are
$0.873402$, $0.885621$, $0.893902$, and $0.899592$.
The dependence of $Y_p$ on $p$ ensures the phase inequality throughout
the entire range of powers.

The remaining exponent consists of
$\psi(tw)-\psi(t)$, $H(tw)/t+w/12$, the $J,W$ terms, and
the Euler--Maclaurin remainder. Its modulus is at most
\begin{equation}\label{eq:edgeeps}
 \epsilon_p=p\left[-\frac23\log(1-U)+\frac{h_*}{N_0}
 +\frac{j_*}{N_0^2}+\frac{w_*}{N_0^3}
 +\frac{C_*(1,0.65)}{N_0^4}\right],
\end{equation}
where
\[
 h_*=\frac{\sqrt2U(1+2U)}{12(1-U)^2},\qquad
 j_*=\frac{2\cdot3\sqrt3U(1+U^2)}{54(1-U)^4},\qquad
 w_*=\frac{13\cdot2^{3/2}\cdot54U(1+U)}{19440(1-U)^3}.
\]
The corresponding change in the normalized integral is bounded by
\begin{equation}\label{eq:Eeps}
 E_{\epsilon,p}=(\e^{\epsilon_p}-1)\nu_{p,+}
                                  \sqrt{\pi/b_p^*}.
\end{equation}
For $p=4$, these inequalities give
$\epsilon_4<0.000841704$ and $E_{\epsilon,4}<0.001201186$.

\subsection{The complementary contour and coefficient signs}
Put $\eta_p=\nu_{p,+}/2$ and
$R_U=(U^3/(1-U^3)+3U/(1-U))/9$.
Use $A_{m,p}$ from \eqref{eq:Am} with $u=U$ and $N_0=124/5$.
The following five regions cover the complement of the two major arcs.
Their normalized bounds, in the same order, are
\begin{align}\label{eq:edgetails}
 T_{1,p}&=8\eta_p\sqrt{N_0}
                 \e^{A_{m,p}-pN_0(c/2-2R_U)},\notag\\
 T_{2,p}&=48\eta_p\sqrt{N_0}
                 \e^{p/3+pM_0/1.3-pN_0G_3},\notag\\
 T_{3,p}&=2\pi\eta_pN_0^{3/2}
                 \e^{p/3+3pM_0-pN_0G_{15}},\notag\\
 T_{4,p}&=16\eta_p\sqrt{N_0}
                 \e^{p/3+3pM_0-pN_0G_7},\notag\\
 T_{5,p}&=2\pi\eta_pN_0^{3/2}
                 \e^{38p/15-pN_0A_0}.
\end{align}
The regions are:
\begin{enumerate}
\item $1\le|v|\le3$ about the nonreal roots;
\item $3\le|v|\le15$ about the nonreal roots, where $|1-q|>1.3$;
\item distance at least $15/N$ from every root, with $|1-q|\ge1/3$;
\item $|\theta|<15/N$, with $|1-q|\ge1/3$;
\item $|1-q|<1/3$.
\end{enumerate}
The first uses the dilogarithm gap, the next three use
\eqref{eq:weighted-gap}, and the last uses
Lemma~\ref{lem:positive-axis}. In the fourth region,
$|1-q|\le1/N+|\theta|$ and $8/N_0<1/3$ imply
$|\theta|\ge7/N$. Each envelope decreases for $N\ge N_0$;
explicitly, $pN_0$ times the corresponding gap exceeds $1/2$
or $3/2$, as dictated by its power of $N$.
For $p=4$, the upper bounds for the five terms are
$0.000000333$, $0.006734345$, $0.016232478$, $0.007875381$,
and $0.013937396$. Before rounding, their sum is less than
$0.044779931$.

To combine the positive real integral with the phases, observe that
\begin{equation}\label{eq:edgephaseinterval}
 -\pi/18\le\psi(t)\le-\pi/18-\log(1-U)/3.
\end{equation}
Let $s_{p,j}$ be a lower bound over this interval for
$g_{p,j}$ if $j=0,2$, and for $-g_{p,1}$ if $j=1$.
Then the desired normalized sign margin is at least
\begin{equation}\label{eq:edgemargin}
 s_{p,j}Q_p-E_{\epsilon,p}-\sum_{k=1}^5T_{k,p}.
\end{equation}
For instance,
$g_{4,2}(t)\ge\cos(4\pi/9-4\log(1-U)/3)>0.17$;
the negative residue-$1$ phase is at least
$-\cos(-8\pi/9-4\log(1-U)/3)$.
Substitution in \eqref{eq:edgemargin} yields the bounds in
Table~\ref{tab:edge}.
\begin{table}[ht]
\centering
\begin{tabular}{crrr}
\toprule
$p$&residue $0$&negative of residue $1$&residue $2$\\\midrule
4&---&0.7571461&0.1022995\\
5&0.5551144&0.8538041&0.2922522\\
6&0.4404906&0.8833014&0.4402989\\
7&0.3035727&0.8781618&0.5722296\\
8&0.1538028&0.8429310&0.6865255\\\bottomrule
\end{tabular}
\caption{Strict lower bounds for normalized sign margins in the
unbounded edge range.}\label{tab:edge}
\end{table}
We have proved the following proposition.
\begin{proposition}\label{prop:edge}
If $t\ge3$, $N=n/t\ge124/5$, and $d=-pn^2F'(t)$ is an integer,
then $c_d^{(p)}(n)/\M_{p,n,d}(t)>0.1$ for $d\equiv2\pmod3$,
and this ratio is less than $-0.1$ for $d\equiv1\pmod3$.
For $p\ge5$, it is also greater than $0.1$ for $d\equiv0\pmod3$.
\end{proposition}
All estimates in this section are uniform on $t\ge3$:
the remaining dependence on $t$ has been bounded in terms of $u\le U$.

\section{Completion of the proof}\label{sec:completion}
\subsection{Finite coefficient identities}
We first record the finite calculations that join the analytic ranges
at their lower ends. They use only integer coefficient identities.
If $a_d$ and $\widetilde a_d$ are the coefficients before and after
multiplication by $(1-q^k)^p$, then
\begin{equation}\label{eq:finite-recurrence}
 \widetilde a_d=\sum_{r=0}^p(-1)^r\binom pr a_{d-rk},
\end{equation}
where coefficients outside the degree range of the former polynomial
are zero. Starting with $a_d=\ind_{d=0}$ and taking
$k=3j-2,3j-1$ in order determines $P_n(q)^p$ exactly.
For a prescribed degree cutoff, the same recurrence applies after
discarding terms above that cutoff.

The stable product supplies another useful recurrence. For coefficients
$a_d^{(p)}$ of
$\prod_{1\le k<L,\ 3\nmid k}(1-q^k)^p$, logarithmic
differentiation gives, for $1\le d<L$,
\begin{equation}\label{eq:stable-recurrence}
 a_0^{(p)}=1,\qquad
 d a_d^{(p)}=-p\sum_{j=1}^d
       \left(\sum_{\substack{k\mid j\\3\nmid k}}k\right)a_{d-j}^{(p)}.
\end{equation}
The restriction $k<L$ is automatic in the inner sum, since $k\le j\le d$.

\begin{lemma}\label{lem:finite}
The following coefficient statements hold.
\begin{enumerate}
\item For $4\le p\le8$ and $1\le n\le300$, there is an integer
$K_{p,n}$ such that
\[
 z_{p,n,m}>0\quad(0\le m\le K_{p,n}),\qquad
 z_{p,n,m}<0\quad(K_{p,n}<m\le pn^2-1).
\]
For $5\le p\le8$ in the same range,
$c_{3m}^{(p)}(n)>0$ for $0\le m\le pn^2$.
\item For $p=4$, $n\ge301$, and $0\le d<900$,
$c_d^{(4)}(n)<0$ if $d\equiv1\pmod3$ and
$c_d^{(4)}(n)>0$ if $d\equiv2\pmod3$.
\item For $5\le p\le8$, $n\ge301$, and $0\le d<1800$,
$c_d^{(p)}(n)$ is positive in residue classes $0,2$ and negative
in residue class $1$.
\end{enumerate}
\end{lemma}
\begin{proof}
Successive substitution in \eqref{eq:finite-recurrence} gives the
finite sign distributions in the following table. Each entry concerns
every coefficient in the indicated residue class and range.
\begin{center}
\begin{tabular}{cccrrr}
\toprule
$p$&$n$&residue&initial sign&zeros&sign changes\\\midrule
4--8&1--300&2&$+$&0&1\\
5--8&1--300&0&$+$&0&0\\\bottomrule
\end{tabular}
\end{center}
For the first row, the total number of target coefficients is
$\sum_{p=4}^8\sum_{n=1}^{300}pn^2=271\,351\,500$.
For the second row it is
$\sum_{p=5}^8\sum_{n=1}^{300}(pn^2+1)=235\,172\,500$.
The accompanying transition tables give $K_{p,n}$ and both adjacent
coefficients for every pair $(p,n)$ in this finite range. The
recurrence determines the whole sequence, including the signs before
and after the two displayed terms.

For $p=4$ and $n\ge301$, the first omitted factor of the finite
product has degree $3n+1\ge904$. Thus, for $d<900$,
\begin{equation}\label{eq:stable4}
 c_d^{(4)}(n)=[q^d]\prod_{\substack{1\le k\le899\\3\nmid k}}
                                      (1-q^k)^4.
\end{equation}
Substitution in \eqref{eq:stable-recurrence} with $L=900$ yields
$a_{3m+1}^{(4)}<0$ and $a_{3m+2}^{(4)}>0$ for $0\le m\le299$.

For the larger powers, apply the truncated recurrence
\eqref{eq:finite-recurrence} through degree $1799$ for each
$301\le n\le600$. The resulting signs are
\begin{equation}\label{eq:prefixsigns}
 c_{3m}^{(p)}(n)>0,\qquad c_{3m+1}^{(p)}(n)<0,\qquad
 c_{3m+2}^{(p)}(n)>0\qquad(0\le m\le599, 5\le p\le8).
\end{equation}
This is a finite set of $4\cdot300\cdot1800=2\,160\,000$ coefficients.
At $n=600$, the first omitted factor has degree $1801$; hence the
coefficients below $1800$ remain unchanged for all subsequent $n$.
Their stable values are also given by \eqref{eq:stable-recurrence}
with $L=1800$. The stable coefficient tables accompany the paper.
This proves all three assertions.
\end{proof}

\subsection{The degree threshold and coverage}
The ordinary saddle satisfies
\begin{equation}\label{eq:threshold}
 d=-pn^2F'(t)=p(n/t)^2\left[\int_0^t\ell(v)\,dv-t\ell(t)\right]
 <pcN^2,\qquad c=\pi^2/27,\qquad N=n/t.
\end{equation}
Here $\int_0^\infty\ell(v)\,dv=c$. The elementary bounds
\begin{equation}\label{eq:cutoffs}
 4c(24.8)^2<900,\qquad 8c(24.8)^2<1800
\end{equation}
show that every ordinary saddle of degree at least $900$ for $p=4$,
or at least $1800$ for $5\le p\le8$, has $N>24.8$.
These are precisely the thresholds needed to apply
Proposition~\ref{prop:edge}; lower degrees are covered by
Lemma~\ref{lem:finite}.

Fix $n\ge301$. We first consider the lower half of the polynomial,
$1\le d\le3pn^2/2$. Strict convexity of $F$ makes the ordinary
saddle map strictly decreasing, and the relevant images are
\begin{equation}\label{eq:coverage}
 \begin{array}{rcl}
 d_{A,p}([0,1])&=&[d_{A,p}(1),3pn^2/2],\\
 d_{B,p}([0.99,3])&=&[d_{B,p}(3),d_{B,p}(0.99)],\\
 d_{B,p}([3,\infty))&=&(0,d_{B,p}(3)].
 \end{array}
\end{equation}
The first two intervals overlap by \eqref{eq:ABoverlap}.
The last two share the ordinary saddle at $t=3$.
Within the last interval, degrees at or above the respective cutoff are
covered by \eqref{eq:threshold}--\eqref{eq:cutoffs} and
Proposition~\ref{prop:edge}; those below it are covered by the
finite identities. Thus every positive degree in the lower half
has been included.

In the compact range, Section~\ref{sec:A} gives the required
residue-$2$ signs whenever $|m-T_{p,n}|>130$.
Proposition~\ref{prop:local} gives the remaining signs outside the
interval of half-width $\varepsilon_p/n^2$.
The intermediate and edge ranges lie below $T_{p,n}-130$ by
\eqref{eq:separation}, and their residue-$2$ coefficients are positive.
The finite low degrees do also: the same bounds give
$301^2\alpha_p-1>730$, while their target indices are at most $599$.
Every residue-$1$ coefficient in the lower half is negative, and
for $p\ge5$ every residue-$0$ coefficient there is positive.
The constant coefficient is $1$.

Reciprocity gives
\begin{equation}\label{eq:reciprocity}
 c_d^{(p)}(n)=c_{3pn^2-d}^{(p)}(n).
\end{equation}
It preserves residue class $0$ and sends upper-half residue $2$
to lower-half residue $1$. Consequently all upper-half residue-$2$
coefficients are negative, and the residue-$0$ coefficients for
$p\ge5$ are positive in the upper half as well. The defining
bounds also give
\[
 (p/2-\alpha_p)301^2-\beta_p301-1>131,
\]
and therefore $T_{p,n}+130<pn^2/2-1$ for $n\ge301$.
Since the coefficient degree is $3m+2$, this places the entire
critical strip in the lower half of the polynomial, and
the upper-half signs agree with \eqref{eq:signminus}.

The interval
$[T_{p,n}-\varepsilon_p/n^2,T_{p,n}+\varepsilon_p/n^2]$
has length at most $3000/n^2<1$. It contains at most one integer.
All earlier terms are positive and all later terms are negative, so
the one possible undetermined term cannot create another sign change.
Finally,
\begin{equation}\label{eq:endpoints}
 z_{p,n,0}=p(p-3)/2>0,\qquad z_{p,n,pn^2-1}=-p<0,
\end{equation}
which ensures that a change does occur. This proves the asserted
sign transition, localization, and bound on zeros for $n\ge301$.
Lemma~\ref{lem:finite} supplies the remaining $n$.

\subsection{The quartic residue-zero coefficients}
It remains only to establish \eqref{eq:r0} for $p=4$.
The Second Borwein Theorem \cite[Theorem 10.3]{WK} gives polynomials
$A,B,C$ with nonnegative coefficients such that
\[
 P_n(q)^2=A(q^3)-qB(q^3)-q^2C(q^3).
\]
The residue-$0$ part of its square, in $x=q^3$, is
$A(x)^2+2xB(x)C(x)$ and consequently has nonnegative coefficients.
Together with the preceding argument, this completes the proof of
Theorem~\ref{thm:main}.

\appendix
\section{Explicit series bounds and coefficient operations}\label{app:series}
We collect the series estimates underlying the explicit inequalities
in the proof. On each closed parameter interval, lower and upper
bounds are carried through the algebraic operations. Positive
denominators are bounded away from zero. In Gaussian estimates the
decay exponent uses the lower bound for the curvature, whereas the
normalizing factor uses its upper bound. Thus every resulting
inequality holds throughout the relevant interval.

\subsection{Analytic Taylor coefficients}
For $t>0.35$, use \eqref{eq:dilog}. If $|u|\le\rho<1$, then
\begin{equation}\label{eq:dilogtail}
 \left|\Li(u)-\sum_{k=1}^K\frac{u^k}{k^2}\right|
 \le\frac{\rho^{K+1}}{(K+1)^2(1-\rho)}.
\end{equation}
Retaining 120 terms suffices for the real-parameter ranges used above;
the complex middle-arc estimate \eqref{eq:localgap} uses 60 terms,
as specified there. Taylor coefficients of $F$ follow from
\begin{equation}\label{eq:Fderivatives}
 tF^{(r)}(t)+rF^{(r-1)}(t)=\ell^{(r-1)}(t)\qquad(r\ge1).
\end{equation}
For $0\le t\le0.35$, use instead the convergent expansion
\begin{equation}\label{eq:Fsmall}
 F(t)=\log3-\frac32t+
 \sum_{k\ge1}\frac{(-1)^{k+1}\zeta(2k)(9^{2k}-3^{2k})}
 {k(2\pi)^{2k}(2k+1)}t^{2k}.
\end{equation}
It follows from the product for $\sinh z/z$ and has radius of
convergence $R=2\pi/9$. Retain the first 75 terms and express
$\zeta(2k)$ by the exact Bernoulli-number formula
\[
 \zeta(2k)=\frac{(-1)^{k+1}B_{2k}(2\pi)^{2k}}{2(2k)!}.
\]
For the $j$th Taylor coefficient on a parameter interval with upper
endpoint $b$, bound the first omitted term in absolute value using
$\zeta(2k)<2$. For $k\ge76$, the ratios of successive positive
majorants are at most
\begin{equation}\label{eq:smallratio}
 (b/R)^2\frac{(2k+2)(2k+1)}{(2k+2-j)(2k+1-j)}<1
\end{equation}
for all derivative orders used in the proof. This bound decreases
with $k$, giving a geometric bound on the tail.

For completeness, the operations on finite Taylor series can be
specified algebraically. If $U(z)=\sum_{j=0}^{L}u_jz^j$ with
$u_0\ne0$, the reciprocal coefficients satisfy
\begin{equation}\label{eq:jetinverse}
 v_0=u_0^{-1},\qquad
 v_k=-u_0^{-1}\sum_{j=1}^ku_jv_{k-j}.
\end{equation}
Products use the usual convolution. If
$\log U(z)=\sum_{j=0}^{L}\ell_jz^j+O(z^{L+1})$, then
\begin{equation}\label{eq:jetlog}
 \ell_0=\log u_0,\qquad
 \ell_k=\frac1k\sum_{j=1}^k j u_jv_{k-j}\quad(k\ge1).
\end{equation}
Exponential coefficients obey the recurrence in \eqref{eq:Br},
and differentiation sends $u_j$ to $(j+1)u_{j+1}$.
Together with \eqref{eq:HJW}, these identities supply every
coefficient of the five analytic functions used in the proof.
The tail bounds are propagated through these operations. The Cauchy
bounds \eqref{eq:compactdisks} and \eqref{eq:scaleddisks} control
the final analytic remainders independently of the retained jets.

\subsection{Elementary quantities and interval subdivisions}
The elementary transcendental quantities are bounded by their
convergent series. For $0\le x\le X$ and an integer $K\ge0$ with
$K+2>X$,
\begin{equation}\label{eq:exptail}
 0\le\e^x-\sum_{j=0}^K\frac{x^j}{j!}
 \le\frac{X^{K+1}}{(K+1)!}\frac1{1-X/(K+2)}.
\end{equation}
For $x>0$, setting $y=(x-1)/(x+1)$ gives
\begin{equation}\label{eq:logtail}
 \left|\log x-2\sum_{j=0}^K\frac{y^{2j+1}}{2j+1}\right|
 \le\frac{2|y|^{2K+3}}{(2K+3)(1-|y|^2)}.
\end{equation}
Taylor's formula bounds the trigonometric functions, and reciprocals
give bounds for negative exponentials. Square roots are bounded by
squaring positive rational endpoints. Rational bounds for $\pi$ may
be obtained, for example, from
$\pi/4=4\arctan(1/5)-\arctan(1/239)$ and the alternating arctangent
series. These formulas, \eqref{eq:dilogtail}, and the Gaussian moment
identities reduce all fixed comparisons above to finite inequalities.

The compact range uses the intervals
$[i/1000,(i+1)/1000]$, $0\le i<1000$; the scaled range uses
$[i/200,(i+1)/200]$, $198\le i<600$.
The finer quartic middle-arc estimate uses
$[k/1000,(k+1)/1000]$, $250\le k<1000$.
All intervals are closed and include both endpoints.
For each positive majorant take upper bounds in the numerator and
lower bounds in positive denominators. At the removable endpoint
$A=0$, the function $H_A=(1-\e^{-A})/A$ in
\eqref{eq:compactgaps} has $H_0=1$; on an interval beginning at
zero one may simply use $H_A\ge\e^{-A}$.

In particular, the error constants in the compact and critical
ranges are obtained in the following order: form the Taylor
coefficients, construct $A_r$ and the coefficient majorants $a_r$,
form $V_s,U_{s+1},Z_s$, apply the explicit Gaussian moments and
tail recurrence, and add the complementary-arc envelopes. The scaled
range follows the same steps with the functions
\eqref{eq:scaledfunctions} and the real-part bound \eqref{eq:F4}.
For unbounded parameters, the remaining envelopes have the form
$x^a\e^{-bx}$. The condition $bx_0>a$ makes each decreasing on
$x\ge x_0$, so its value at $x_0$ bounds the entire remaining range.

\end{document}